\documentclass[12pt]{amsart}
\usepackage[english]{babel}
\usepackage{paralist,url,verbatim, anysize}
\usepackage{amscd}
\usepackage[all,cmtip]{xy} 
\usepackage{amsmath,amsthm,amssymb,amsfonts}
\usepackage[bookmarksopen=true]{hyperref}
\usepackage[usenames, dvipsnames]{color}
\usepackage{graphicx}
\usepackage{hyperref}
\usepackage{euler, amsfonts, amssymb, latexsym, epsfig,epic}
\usepackage{xcolor}

\makeatletter
\let\@fnsymbol\@arabic
\makeatother

\theoremstyle{plain}
\newtheorem{theorem}{\bf Theorem}[section]

\newtheorem{lemma}[theorem]{Lemma}
\newtheorem{proposition}[theorem]{Proposition}

\theoremstyle{definition}

\newtheorem{example}[theorem]{Example}
\newtheorem{remark}[theorem]{Remark}

\newtheorem*{theorem*}{\bf Theorem}

\newcommand{\ini}{\operatorname{in} }

\newcommand{\NN}{\mathbb{N}}

\newcommand{\Lex}{\operatorname{Lex} }

\newcommand{\height}{\operatorname{height} }

\newcommand{\mm}{\mathfrak{m}}

\newcommand{\reg}{\operatorname{reg} }

\definecolor{mypink}{RGB}{215, 5, 234}

\newcommand{\projdim}{\mathrm{projdim}}

\begin{document}

\title{Regularity  of primes associated with polynomial parametrisations} 

\author[F.Cioffi]{Francesca Cioffi}
\address{Dipartimento di Matematica e Applicazioni dell'Universit\`{a} di Napoli Federico II\\ via Cintia, 80126  Napoli, Italy
         }
\email{{cioffifr@unina.it}}

\author[A. Conca]{Aldo Conca}
\address{Dipartimento di Matematica dell'Universit\`{a} di Genova\\ 
         Via Dodecaneso 35, 
         16146 Genova, Italy}
\email{{conca@dima.unige.it}}

\begin{abstract} 
We prove a  doubly exponential  bound for  the  Castelnuovo-Mumford regularity  of prime ideals   defining  varieties with polynomial parametrisation.  
\end{abstract}

 \thanks{Both authors are supported by INdAM-GNSAGA. MSC classification: 13D02, 13P10}

\maketitle

 \section{Introduction} 
 Let $I$ be a homogeneous ideal in the polynomial ring $R=K[x_1,\dots, x_n]$ over a field~$K$. The Castelnuovo-Mumford regularity $\reg(I)$ of $I$ has been introduced in \cite{EG} and \cite{O} as an algebraic counterpart of the corresponding  notion introduced by Mumford \cite{M} for coherent sheaves over projective spaces. It quickly became one of the most important homological and cohomological invariants of $I$. It is defined in terms of the graded Betti numbers $\beta_{i,j}(I)$ as 
$$\reg(I)=\max\{ j-i :  \beta_{i,j}(I)\neq 0\}$$ 
as well as in terms of the graded local cohomology modules $H_\mm^i(I)$ as 
$$\reg(I)=\max\{ j+i :  H_\mm^i(I)_j\neq 0\},$$ 
where $\mm=(x_1,\dots, x_n)$. 
It is known \cite{BM,CS} that $\reg(I)\leq (2u)^{2^{n-2}}$, where $u$ is the largest degree of a generator of $I$. When $P$ is a prime homogeneous ideal one expects better bounds for $\reg(P)$,  see \cite{BM} and the recent \cite{MP2} for an overview. On the other hand, McCullough  and Peeva   proved in \cite{MP} that $\reg(P)$  cannot be bounded above by any polynomial in the degree (or multiplicity)   $\deg(R/P)$ hence disproving the long--standing  Eisenbud-Goto conjecture \cite{EG}. Later on  Caviglia,  Chardin,  McCullough,  Peeva and Varbaro \cite{CCMPV} proved that   $\reg(P)$ can be actually bounded by (a highly exponential) function in $\deg(R/P)$.

 In this short note we prove a  doubly exponential   bound for $\reg(P)$ when $P$   defines a  variety with a polynomial parametrisation that does not involve the degree of the generators of $P$  but  only the numerical data of the parametrisation.     
 
 \begin{theorem} 
 \label{main} 
 Let $P$ be the kernel of a $K$-algebra map $\phi: K[x_1,\dots, x_n] \to K[y_1,\dots, y_m]$ with $\phi(x_i)=f_i$ homogeneous polynomials of degree $d>0$.  Then one has: 
 $$\reg(P)\leq d^{n2^{m-1}-1}.$$
 \end{theorem}
 
 We remark that a ``combinatorial" bound for  $\reg(P)$ in the case of  curves (i.e. $m=2$) with a monomial parametrisation  has been obtained in \cite[Prop.~5.5]{L} and that it would be very interesting to obtain similar bounds for  higher dimensional toric ideals.  We thank two anonymous referees for helpful comments and for suggesting that  a bound comparable with  that of Theorem \ref{main} might be obtained  using the techniques and the ideas of \cite{CCMPV} combined with other estimates on the regularity and on the degree.

\section{Flat extensions, regularity and elimination} 

For the proof of  Theorem \ref{main} we  need to collect some ingredients. 

\subsection{Regularity and flat extensions}
 Let  $R=K[x_1,\dots, x_n]$ with its standard graded structure.  Let  $d$ be a positive integer and let  $\alpha:R\to R$ be the $K$-algebra map defined by  $\alpha(x_i)=x_i^d$ for $i=1,\dots, n$.  
For a homogeneous ideal $I$ of $S$ we set $I'= \alpha(I)R$. By construction $I'$ is homogeneous and we have: 

\begin{lemma} 
\label{regflat} 
 $\reg(I)\leq  \reg(I')/d.$
\end{lemma} 
\begin{proof} 
By \cite{H}  the map $\alpha$ is flat. Hence  if $F$ is a minimal graded free resolution of $I$   then $\alpha(F)R$ is a  minimal graded free resolution of $I'$.  Therefore the graded  Betti numbers of $I$ and $I'$  are related as follows: $\beta_{i,jd}(I')=\beta_{i,j}(I)$ for all $i,j$ and $\beta_{i,j}(I')=0$ if $d$ does not divide $j$.  For $i=0,\dots, \projdim(I)$  set $t_i(I)=\max\{ j : \beta_{i,j}(I)\neq 0\}$. Then we have  $t_i(I')=dt_i(I)$. By definition  $\reg(I)=\max\{ t_i(I)-i :  i=0,\dots, \projdim(I)\}$. Let $p$ be the largest integer $i$ such that $\reg(I)=t_i(I)-i$. 
Then    
\begin{equation*}
\reg(I')\geq t_p(I')-p=dt_p(I)-p =d(\reg(I)+p)-p=d\reg(I)+p(d-1).
\end{equation*} 
It follows that 
\begin{equation}
\label{eq1} 
\reg(I')/d\geq \reg(I)+p(d-1)/d
\end{equation} 
and, in particular, 
$$\reg(I')/d\geq \reg(I).$$
 \end{proof} 
 
 \begin{remark} The proof of Lemma \ref{regflat} shows that  the inequality in Lemma \ref{regflat} is strict unless $d=1$ (which is obvious) or the index $p$ defined in the proof is $0$ and this happens only if $I$ is principal.  
If $R/I$ is Cohen-Macaulay then $p=\height(I)-1$ and in Eq.~\eqref{eq1} one has equality. For example if $I$ is a complete intersection of $c$ forms of degree $s$ then $\reg(I)=sc-(c-1)$ and $\reg(I')=dsc-(c-1)$ so that  
 $$\reg(I')/d-\reg(I)=(c-1)(d-1)/d.$$
 In general however in  Eq.~\eqref{eq1} one does not have equality. For example for an ideal $I$ with $\projdim(I)=4$ and 
 \begin{equation}
 \label{seqt} 
 (t_0(I), t_1(I), t_2(I), t_3(I), t_4(I))=(2, 4, 5, 5, 6)
 \end{equation} 
  one has $\reg(I)=3$, $p=2$ and  $\reg(I')=6d-4$ for $d>1$.  Hence  the inequality in Eq.~\eqref{eq1} is strict for $d>2$. An ideal with   invariants as in \eqref{seqt} is, for example, $I=(x_1)(x_1,x_2,\dots, x_5)+(x_2^2,x_3^2)$. 
 \end{remark} 
   
  \subsection{Flat extensions and elimination}
Now let $\ell\geq n$ and $S=K[x_1,\dots,x_\ell]$. Let $J$ be an ideal  of $S$  and $I=J\cap R$. 
 Let $d_1,\dots, d_\ell\in \NN_{>0}$ and $\varphi: S \to S$ be the $K$-algebra map defined by $\varphi(x_i)=x_i^{d_i}$ and let $\alpha$ be the restriction of $\varphi$ to $R$.  Let  $J$ be an ideal of $S$ and   $I=J\cap  R$.   Set 
 $$I'=\alpha(I)R \mbox{  and }J'=\varphi(J)S.$$
 
   Let $<$  be the lexicographic order  associated with  $x_\ell>x_{\ell-1}>\dots >x_1$. Recall that $<$ is an elimination  term order for the variables $x_{n+1}, \dots, x_\ell$. In particular, 
if $G$ is a Gr\"obner basis of $J$ with respect to $<$, then $G\cap R$ is a Gr\"obner basis of $I$ with respect to $<$ restricted to  $R$, see for example \cite[Thm. 2, Sect. 3.1]{CLO}.  

\begin{lemma}\label{poweli}
With the notation above we have: 
\begin{itemize}
\item[{\rm (i)}] If $G$ is a Gr\"obner basis of $J$ with respect to $<$, then  $\varphi(G)=\{ \varphi(g) : g\in G\}$ is a Gr\"obner basis of $J'$ with respect to $<$. 
\item[{\rm (ii)}] $I'=J'\cap R$.
\end{itemize}
\end{lemma}

\begin{proof}
Firstly we observe that, since we deal with the lex order,  for every pair  of monomials $\tau$ and $\sigma$ of $S$, we have $\tau < \sigma$ if and only if $\varphi(\tau)<\varphi(\sigma)$. In particular, 
$\ini(\varphi(f))=\varphi(\ini(f))$ for any non-zero $f\in S$.  
Secondly,   we  observe that $\varphi$ is compatible with the Buchberger criterion.  Indeed denoting  by $S(f,h)$ the $S$-polynomial of two polynomials $f$ and $h$ of $S$, one has that $\varphi(S(f,h))=S(\varphi(f), \varphi(h))$. Furthermore  if $h=\sum_{g\in G}  p_g g$ is a division with remainder $0$ of $h$ with respect to $G$, then $ \varphi(h)=\sum_{g\in G}  \varphi(p_g)\varphi(g)$ is a division with remainder $0$ of $\varphi(h)$  with respect to $\varphi(G)$. By the Buchberger criterion \cite{E}, this is enough to conclude that  (i) holds. 

To prove (ii)  we observe that  $I$ is generated by $G\cap R$ so that $I'$ is generated by the set  $\alpha(G \cap R)$.  On the other hand, by (i),  $J'\cap R$ is generated by the set 
 $\varphi(G) \cap R$.  Clearly  $\alpha(G \cap R)=\varphi(G) \cap R$ and hence it follows  that $I'=J'\cap R$.  
\end{proof}

\subsection{Regularity and elimination} 
We keep the notation above, i.e.~$\ell\geq n$ and  $S=K[x_1,\dots,x_\ell]\supseteq R=K[x_1,\dots, x_n]$. Let $J$ be a homogeneous ideal of $S$ and $I=J\cap R$. In general  it can happen that   $\reg(I)>\reg(J)$ or   $\reg(I)<\reg(J)$. 

\begin{example} 
Let  $J=(x_1^2, x_2^2)$ and $I=J\cap K[x_1]=(x_1^2)$.  Then  $\reg(J)=3$  and $\reg(I)= 2$. 
Let $J=(x_1x_2 + x_2x_3, x_1x_3, x_3^2)$ and  $I=J\cap K[x_1,x_2]= (x_1^2x_2)$.  Then  $\reg(J)=2$  and $\reg(I)= 3$. 
\end{example} 

However we can get a bound for  $\reg(I)$ in terms of  the regularity of the lexicographic ideal $\Lex(J)$ 
of the ideal $J$.  We refer the reader to  \cite{BH} for generalities on the  lexicographic ideal $\Lex(J)$. Here we just recall that $\Lex(J)$ is defined as $\oplus_{i\in \NN} \Lex(J_i)$ where $\Lex(J_i)$ is the $K$-vector space generated by the largest  $\dim J_i$ monomials of degree $i$ with respect to the lexicographic order. Macaulay proved that $\Lex(J)$ is actually an ideal of $S$ and Bigatti~\cite{B}, Hulett \cite{Hu} and Pardue \cite{P} proved that $\beta_{ij}(J)\leq \beta_{ij}(\Lex(J))$ for all $i,j$. In particular one has $\reg(J)\leq \reg (\Lex(J))$. 

\begin{proposition} 
\label{regbound} 
With the notations above we have  $\reg(I)\leq \reg(\Lex(J))$.
\end{proposition} 

\begin{proof}  Let $G$ be a   Gr\"obner basis of  $J$ with respect to the lex order with $x_\ell>x_{\ell-1}>\dots>x_1$. The elements of $G$ that belong to $R$  form a Gr\"obner basis  of  $I$. In particular $\ini(I)=\ini(J)\cap R$. It is  known that  $\reg(I)\leq \reg(\ini(I))$ \cite{E}.  Furthermore  by \cite[Cor.~2.5]{OHH} we have $\reg(\ini(I))\leq \reg(\ini(J))$ since $R/\ini(I)$ is an algebra retract of $S/\ini(J)$. Finally  $\Lex(J)=\Lex(\ini(J))$ because $J$ and $\ini(J)$ have the same Hilbert function and $\reg(\ini(J))\leq \reg(\Lex(J))$ by the Bigatti-Hulett-Pardue theorem mentioned above. Summing up, 
$$\reg(I)\leq \reg(\ini(I))\leq \reg(\ini(J))\leq \reg(\Lex(J))$$
and we are done. 
\end{proof}

\section{Proof of the main Theorem}
We have collected all the ingredients for the proof of Theorem~\ref{main}.


\begin{proof}[Proof of  Theorem \ref{main}] One has $P=J \cap K[x_1,\dots,x_n]$ 
with $J=(x_i-f_i : i=1,\dots, n)\subset S=K[x_1,\dots, x_n, y_1,\dots, y_m]$. Consider   $\varphi:S\to S$   defined by $\varphi(x_i)=x_i^d$  and $\varphi(y_i)=y_i$ and let $\alpha: R\to R$ be the restriction of $\varphi$ to $R$, so that $\alpha(x_i)=x_i^d$. 

In this setting, we have $J'=\varphi(J)S=(x_i^d-f_i : i=1,\dots, n)$ and $P'=\alpha(P)R$. The ideal $J'$ is   a complete intersection of $n$ forms of degree $d$ and dimension $m$ so its Hilbert function and hence $\Lex(J')$  just depend on $n,d,m$. 
Set $G_{n,d,m}=\reg(\Lex(J'))$. 
 
By  Lemma \ref{poweli}(ii) we have $P'=J' \cap R$ and thanks to Proposition \ref{regbound}: 
$$ \reg(P')\leq G_{n,d,m}.
$$
Since by Lemma \ref{regflat} $\reg(P)\leq \reg(P')/d$ we conclude
$$\reg(P)\leq G_{n,d,m}/d.$$
Taking into account that by \cite[Corollary 3.4]{CM} we have  $G_{n,d,m}\leq d^{n2^{m-1}}$,   we obtain the desired inequality. 
\end{proof}


\begin{thebibliography}{10}

\bibitem{BM}
Dave Bayer and David Mumford, \emph{What can be computed in algebraic
  geometry?}, Computational algebraic geometry and commutative algebra
  ({C}ortona, 1991), Sympos. Math., XXXIV, Cambridge Univ. Press, Cambridge,
  1993, pp.~1--48.

\bibitem{B}
Anna~Maria Bigatti, \emph{Upper bounds for the {B}etti numbers of a given
  {H}ilbert function}, Comm. Algebra \textbf{21} (1993), no.~7, 2317--2334.

\bibitem{BH}
Winfried Bruns and J\"{u}rgen Herzog, \emph{Cohen-{M}acaulay rings}, Cambridge
  Studies in Advanced Mathematics, vol.~39, Cambridge University Press,
  Cambridge, 1993. 

\bibitem{CS}
Giulio Caviglia and Enrico Sbarra, \emph{Characteristic-free bounds for the
  {C}astelnuovo-{M}umford regularity}, Compos. Math. \textbf{141} (2005),
  no.~6, 1365--1373.

\bibitem{CCMPV}
Giulio Caviglia, Marc Chardin, Jason McCullough, Irena Peeva, and Matteo
  Varbaro, \emph{Regularity of prime ideals}, Math. Z. \textbf{291} (2019),
  no.~1-2, 421--435.

\bibitem{CM}
Marc Chardin and Guillermo Moreno-Soc\'{\i}as, \emph{Regularity of lex-segment
  ideals: some closed formulas and applications}, Proc. Amer. Math. Soc.
  \textbf{131} (2003), no.~4, 1093--1102.

\bibitem{CLO}
David Cox, John Little, and Donal O'Shea, \emph{Ideals, varieties, and
  algorithms}, second ed., Undergraduate Texts in Mathematics, Springer-Verlag,
  New York, 1997, An introduction to computational algebraic geometry and
  commutative algebra.

\bibitem{E}
David Eisenbud, \emph{Commutative algebra}, Graduate Texts in Mathematics, vol.
  150, Springer-Verlag, New York, 1995, With a view toward algebraic geometry.

\bibitem{EG}
David Eisenbud and Shiro Goto, \emph{Linear free resolutions and minimal
  multiplicity}, J. Algebra \textbf{88} (1984), no.~1, 89--133.

\bibitem{H}
Robin Hartshorne, \emph{A property of {$A$}-sequences}, Bull. Soc. Math. France
  \textbf{94} (1966), 61--65.

\bibitem{Hu}
Heather~A. Hulett, \emph{Maximum {B}etti numbers of homogeneous ideals with a
  given {H}ilbert function}, Comm. Algebra \textbf{21} (1993), no.~7,
  2335--2350.

\bibitem{L}
S.~L$'$vovsky, \emph{On inflection points, monomial curves, and hypersurfaces
  containing projective curves}, Math. Ann. \textbf{306} (1996), no.~4,
  719--735.

\bibitem{MP}
Jason McCullough and Irena Peeva, \emph{Counterexamples to the
  {E}isenbud-{G}oto regularity conjecture}, J. Amer. Math. Soc. \textbf{31}
  (2018), no.~2, 473--496.

\bibitem{MP2}
\bysame, \emph{The regularity conjecture for prime ideals in polynomial rings},
  EMS Surv. Math. Sci. \textbf{7} (2020), no.~1, 173--206.

\bibitem{M}
David Mumford, \emph{Lectures on curves on an algebraic surface}, Annals of
  Mathematics Studies, No. 59, Princeton University Press, Princeton, N.J.,
  1966, With a section by G. M. Bergman.

\bibitem{OHH}
Hidefumi Ohsugi, J\"{u}rgen Herzog, and Takayuki Hibi, \emph{Combinatorial pure
  subrings}, Osaka J. Math. \textbf{37} (2000), no.~3, 745--757.

\bibitem{O}
Akira Ooishi, \emph{Castelnuovo's regularity of graded rings and modules},
  Hiroshima Math. J. \textbf{12} (1982), no.~3, 627--644.

\bibitem{P}
Keith Pardue, \emph{Deformation classes of graded modules and maximal {B}etti
  numbers}, Illinois J. Math. \textbf{40} (1996), no.~4, 564--585.

\end{thebibliography}

\providecommand{\bysame}{\leavevmode\hbox to3em{\hrulefill}\thinspace}
\providecommand{\MR}{\relax\ifhmode\unskip\space\fi MR }
\providecommand{\MRhref}[2]{%
  \href{http://www.ams.org/mathscinet-getitem?mr=#1}{#2}
}
\providecommand{\href}[2]{#2}

\end{document}